\documentclass{article}
\usepackage[utf8]{inputenc}
\usepackage{igor-commands}
\usepackage[capitalise]{cleveref}
\usepackage{authblk}

\usepackage[ shadow, colorinlistoftodos,textsize=tiny, obeyFinal]{todonotes}
\setuptodonotes{fancyline, backgroundcolor=gray!10,bordercolor=gray}

\newtheorem{theorem}{Theorem}[section]
\newtheorem{lemma}[theorem]{Lemma}
\newtheorem{defi}[theorem]{Definition}
\newtheorem{prop}[theorem]{Proposition}
\newtheorem{coro}[theorem]{Corollary}

\theoremstyle{definition}
\newtheorem{Example}[theorem]{Example}
\newtheorem{remark}[theorem]{Remark}

\newcommand{\G}{\Gamma}
\newcommand{\Gab}{G^{\textrm{Ab}}}
\newcommand{\Ab}{\textrm{Ab}}

\newcommand{\Commd}{\text{Comm}_d} 
\newcommand{\Comm}{\text{Comm}_b}
\newcommand{\CComm}{\text{Comm}}
\newcommand{\nbhd}[1]{B(#1)} 
\DeclareMathOperator{\St}{St}

\newcommand{\bbe}{\mathbf e}

\newcommand{\bx}{\mathbf x}
\newcommand{\by}{\mathbf y}
\newcommand{\bv}{\mathbf v}

\title{Graph Powers of Groups}

\author{Gabe Cunningham\footnote{School of Computing and Data Science, 
Wentworth Institute of Technology, 
550 Huntington Avenue, 
Boson, MA 02215 (USA)\\
email: cunninghamg1@wit.edu}\; and Igor Minevich\footnote{same address, email: minevichi@wit.edu}}

\affil{Wentworth Institute of Technology}

\begin{document}


\maketitle

\begin{abstract}
The Lights Out Puzzle, played on a graph $\Gamma$, has been studied using linear algebra over $\mathbb{F}_2$ and more generally over $\mathbb{Z}/k\mathbb{Z}$. We generalize the setting by allowing the states of vertices to be the elements of a group $G$, where a \textit{click} in vertex $v$ multiplies the state of $v$ and its neighbors by an element $g \in G$ on the right. 

Starting with the identity element $e \in G$ for all vertices, the totality of all achievable state configurations forms a group $G^{\Gamma}$. This group generalizes parallel products of group actions and provides a rich structure for analysis. For many graphs, which we term ``RA'' (reducible to abelian), the problem reduces --- regardless of $G$ --- to a linear algebra question over $\mathbb{Z}$. We discuss a chain of five different subgroups consisting of commutators and introduce techniques for showing that families of graphs are RA using each. In particular, using Heisenberg groups, we establish that a graph is RA precisely when a certain lattice spans $\mathbb{Z}^{|\Gamma|}$. While most graphs appear to be RA, we show the odd-dimensional cube graphs $Q_{2n+1}$ and folded cube graphs $\square_d$, for $d$ odd or 2, are not.
\end{abstract}

\section{Introduction}

The classic Lights Out Puzzle is played on a grid of squares that can be either on or off, and clicking a square makes it and the adjacent squares switch states. The goal is to start with all squares being on and get to a state where all are off. Equivalently, we could start with all the squares being in state 0 and thinking about clicking like adding 1 to the state of a vertex and its neighbors $\pmod 2$, with the goal being to get to all 1's.


We explore a natural generalization of the Lights Out Puzzle, or Confused Electrician Games \cite{electrician}, by changing the geography of the puzzle as well as the number and quality of states that a given puzzle piece can experience. Instead of a 5x5 board, we have a finite graph $\G$ whose vertices are the placeholders substituting what used to be squares. And now, the vertices need not simply have two states (modeled by the group of two elements, $\Z/2\Z$). Instead, they are colored according to the set of all the elements of \textit{any} given group $G$, and their states change accordingly. A ``click'' in a vertex must specify an element $g$ of the group $G$ that is performed on the vertex and its neighbors, and then the states of the vertex and its neighbors are multiplied on the right by $g$.

As it turns out, the set of all possible states that could be achieved simultaneously in all the vertices, starting from the identity element $1 \in G$ in all of them, is a group. 
If we start with any other state, the states we can reach from there are in direct 1-1 correspondence with this group, which we call $G^\G$. 
Many authors (far too many to list) have analyzed this kind of situation in the case of abelian groups $G$ (see \cite{electrician, LOFiniteGraphs,  Eriksson2001note,  Fleischer2013survey, Giffen2009generalizing, Kreh2017lights, MUETZE20072755, Sutner2000sigma, Torrence}), so we are most interested in the non-abelian case. 

The short exact sequence
\begin{equation}
    \label{main-short-exact-seq}
    1 \to [G, G]^{|\G|} \cap G^\G \to G^\G \to (G^\Ab)^\G \to 1
\end{equation}
turns the exploration of $G^\G$ into (1) an ``abelian question'' (what is $(G^\Ab)^\G$), for which we provide some answers in \cref{sec:abelian} but for the most part defer to other papers, and (2) the amazing exploration of $[G, G]^{|\G|} \cap G^\G$. If the latter is all of $[G,G]^{|\G|}$, then we can change the state of any vertex by any commutator independent of the rest, which reduces the question completely to the abelian case, making the description of $G^\G$ straightforward. Whenever $[G,G]^{|\G|} \sub G^\G$ for all $G$ (and a fixed $\G$), we say that $\G$ is \emph{RA (reducible to abelian)}.

Indeed, we show many graphs are RA (including grids, trees, cycles, complete bipartite graphs, and more) -- see Section \ref{RA-families}. We also describe a method for directly computing whether a graph is RA; see Theorem~\ref{RA-iff}. The problem of constructing families of non-RA graphs seems difficult, though we exhibit two families of non-RA graphs; see Corollary~\ref{graphs-not-RA}. 

For some families of graphs, it is easy to show that they are RA using a subgroup of $[G,G]^{|\G|} \cap G^\G$ that is easy to work with, namely the subgroup generated by commutators of clicks. There are RA graphs that cannot be proven to be RA using only the commutator tricks of Section~\ref{RA-families}, and we will explore these in \cite{FullPaper}. 
Here we begin to explore the chain of inclusions of subgroups of $[G,G]^{|\G|} \cap G^\G$ (see \eqref{all-layers}). The structure of Heisenberg groups of order $p^3$ is so nice that these subgroups turn out to coincide for every $p$, and we are able to use a local-to-global principle to make the characterization of RA graphs in Theorem~\ref{RA-iff}. 

Our work sheds light on subgroups of direct products of groups and provides heretofore unexplored generalizations of topics related to the Lights Out Puzzle. It also leads naturally to a generalization of the \emph{parallel product} (see \cite{parallel-product}) of group actions. In the examples in Section~\ref{subsect:DefEg} and Example~\ref{eg:Rubik_full}, we briefly explore one situation where this sort of action may arise, namely when we have  Rubik's Cubes strung together in a way that when you move one of them, the adjacent ones also move in the same way. 


\section{Preliminaries}

    \subsection{Graph theory}

    All graphs in this paper will be finite undirected simple graphs (no loops or multiple edges). If $v$ is a vertex of a graph $\Gamma$, then $\nbhd{v}$ will be the closed neighborhood of $v$; that is, the set that contains $v$ and all of its neighbors in $\Gamma$. 
    The number of vertices of $\Gamma$ will be $|\Gamma|$.

    Unless explicitly stated otherwise, we will assume the vertices of the graph are labeled $v_1, v_2, \dots, v_n$ (with $n = |\Gamma|$) and the notation $\bbe_i$ refers to the $|\Gamma|$-dimensional  vector $\langle 0, \dots, 0, 1, 0, \dots, 0\rangle$ with the $1$ in the $i$-th place, corresponding to $v_i$. More generally, if $v$ is a vertex of $\G$, then $\bbe_v$ is the vector with a 1 in the coordinate corresponding to $v$ and a 0 everywhere else.

    By the \emph{indicator vector} for a subset $U \subseteq V(\G)$, we mean $\sum_{v \in U} \bbe_v$, i.e. the vector with a 1 in each coordinate that corresponds to a member of $U$.

    We will refer to the following standard graphs:
    \begin{enumerate}
        \item The path graph $P_n$ with $n$ vertices,
        \item The cycle graph $C_n$ with $n$ vertices,
        \item The complete graph $K_n$ with $n$ vertices,
        \item The complete bipartite graph $K_{m,n}$ with $m+n$ vertices, 
        \item The star graph $\St_n$ with $n+1$ vertices, and
        \item The $d$-dimensional cube graph $Q_d$ with $2^d$ vertices, represented by bit strings of length $d$, with two vertices adjacent if and only if they differ in a single bit.
    \end{enumerate}


    The adjacency matrix of $\G$ will be denoted $\Adj(\G)$. 

    \subsection{Group theory} \label{sec:gp-theory}

    We will denote the identity of a group by $1$ and abbreviate $\Z/n\Z$ to $\Z_n$. The commutator of $x, y \in G$ is $[x, y] = xyx^{-1}y^{-1}$. 
    We will denote the abelianization of $G$ (that is, $G / [G, G]$) by $G^\Ab$.


    We will denote by $D_{2n}$ the dihedral group (of size $2n$) on $n$ vertices and use two presentations for it: 
    \begin{align*}
    D_{2n} &= \langle x, y \mid x^2 = y^2 = (xy)^n = 1\rangle \\
    &= \langle r, s \mid s^2 = (rs)^2 = r^n = 1\rangle.
    \end{align*}
    
    Recall (see e.g. \cite{DummitandFoote}) for any prime $p$, the Heisenberg group is the group of $3 \times 3$ upper triangular matrices in $\GL_3(\F_p)$ with 1's on the main diagonal, which we will refer to as $H(\F_p) = H_1(\F_p)$; more generally, Heisenberg groups $H_n(\F_p)$ of order $p^{2n+1}$ exist, whose elements are $(n+2) \times (n+2)$ matrices of the form
    \[\matnine1{\bf a}c0{I_n}{\bf b}001.\]
    By explicit computation, it is easy to deduce several key properties:
    \begin{enumerate}
        \item All such groups have exponent $p$ if $p > 2$, and if $p = 2$, the exponent is 4.
        \item For the special case where $p = 2$ and $n = 1$, we have $H(\F_2) \cong D_8$.
        \item The commutator subgroup of any Heisenberg group is the center, namely the subgroup (isomorphic to) $\Z/p\Z$ where $\bf a = \bf b = \bf 0$.
        \item The abelianization of $H_n(\F_p)$ is $(\Z/p\Z)^{2n}$, and we can find the $2n$ generators for the direct summands by taking $\bf a = \bf e_i$ and $\bf b = \bf 0$ and vice-versa, with $c = 1$; such generators are involutions when $p = 2$.
    \end{enumerate} 

    \subsection{Linear algebra}
    \label{sec:linear-algebra}

    The Smith Normal Form \cite{DummitandFoote} of an $m \times n$ matrix $M$ with integer coefficients is the diagonal matrix $N$ with integer entries such that there exist matrices $U \in GL_m(\Z), V \in GL_n(\Z)$ with $UMV = N$ and the entries along the diagonal are the \textbf{elementary divisors} (or invariant factors) $a_1, a_2, \dots, a_{\min(m, n)}$ of $M$, with $a_1 \mid a_2 \mid a_3 \mid \cdots$. It is well-known that the $i$-th elementary divisor equals $d_i(M)/d_{i-1}(M)$ \cite{SNF}, where $d_i(M)$ is the greatest common divisor of the $i \times i$ minors of $M$ and $d_0(M) = 1$.
        
    We briefly mention one useful fact about elementary divisors.

    \begin{prop}
    \label{row-sums-fund-factors}
    Suppose $M$ is $m \times n$ with $m \ge n$. If there is a prime $p$ such that every row sum of $M$ is divisible by $p$, then $M$ has an elementary divisor that is divisible by $p$ (possibly equal to 0).
    \end{prop}

    \begin{proof}
    The submatrix corresponding to any $m \times m$ minor multiplied by $\vec 1 = (1, 1, \dots, 1)^T$ results in $\vec 0 \pmod p$, which means it is a singular matrix $\pmod p$, hence its determinant is divisible by $p$. Thus, $p \mid d_m(M)$, and since $d_0(M) = 1$ and each $d_i(M)$ divides $d_{i+1}(M)$, there is an index $i$ somewhere along the chain such that $p \nmid d_{i}(M)$ but $p \mid d_{i+1}(M)$, and for that $i$ we have the $i$-th elementary divisor $a_i$ divisible by $p$.
    \end{proof}
                            
\section{Graph powers of groups}

\subsection{Definition and examples}
    \label{subsect:DefEg}
    The guiding metaphor for what we will call the \emph{graph power of a group} is this: start with a graph $\Gamma$ and a group $G$. Put the identity element of $G$ at every vertex of $\Gamma$ -- this corresponds with the identity element of $G^n$. If we `click' vertex $v$ with group element $g$, then we want to multiply the states of all vertices in the closed neighborhood $\nbhd{v}$ by $g$ on the right. The set of achievable states is what we call $G^\Gamma$. 

    We will start with a somewhat more general way to construct subgroups of $G^n$. First, we make the following definition.

    \begin{defi}
    Let $G$ be a group, and let $\bx = (x_1, \ldots, x_n) \in \Z^n$. 
    We define 
    \[ g^{\bx} = (g^{x_1}, \ldots, g^{x_n}) \in G^n. \]
    \end{defi}

    The following properties are clear:

    \begin{prop} \label{basic-props}
    \begin{enumerate}[(a)]
        \item $g^{\bx} g^{\by} = g^{\bx+\by}$
        \item If $k \in \Z$, then $(g^k)^{\bx} = g^{k\bx}$
        \item If every component of $\bx$ is 0 or 1, then $g^\bx h^\bx = (gh)^\bx$
    \end{enumerate}
    \end{prop}

    \begin{defi} \label{def:g-to-the-s} 
    Suppose $S \subseteq \Z^n$. We define 
    \[ G^S = \langle g^{\bx} \mid g \in G, \bx \in S \rangle. \]
    If $M$ is an integer matrix, then we define $G^M$ to be $G^S$ where $S$ is the set of rows of $M$. \end{defi}

    Note that by Proposition~\ref{basic-props}, $G^S = G^{\langle S \rangle}$, so we will usually think of $S$ as a set of generating vectors, rather than an entire integer lattice.

    Now suppose that $\Gamma$ is a graph with vertex set $\{v_1, \ldots, v_n\}$.
    We define the \emph{activation matrix of $\G$} to be
    \[ A_\Gamma = \Adj(\Gamma) + I_n,\]
    where $I_n$ is the $n \times n$ identity matrix. 
    If $v$ is a vertex of $\Gamma$ and $\bx$ is the corresponding row of $A_\Gamma$, then `clicking $v$ with $g$' is accomplished by multiplying the state by $g^{\bx}$. In other words, our metaphor is realized with the following definition:
    
    \begin{defi}
    If $\Gamma$ is a finite graph with $n$ vertices, then the \emph{graph power of $G$ with respect to $\G$} (or simply the \emph{$\G$ power of $G$}),
    denoted $G^\Gamma$, is the subgroup of $G^n$ defined by $G^\Gamma = G^{A_{\Gamma}}$, where $A_{\G} = \Adj(\G) + I_n$.
    \end{defi}

    For convenience, if $v$ is a vertex of $\G$, then we will write $g^v$ to mean $g^\bx$, where $\bx$ is the row of $A_\G$ that corresponds to $v$. Similarly, if $k$ is an integer, we understand $g^{kv}$ to mean $(g^k)^v$. Recall that $\bbe_v$ is the vector with a 1 in the coordinate corresponding to $v$ and 0 everywhere else, which implies that $g^{\bbe_v}$ is the element of $G^n$ that has a $g$ in the coordinate corresponding to $v$ and a $1$ (i.e. the identity of $G$) everywhere else.
    
    We will continue to use the Lights Out metaphor, so that multiplying (on the right) by $g^v$ will also be referred to as \emph{clicking $v$ with $g$}. The utility of this metaphor will become apparent.

    For GAP code that can be used to compute $G^\G$, see Figure~\ref{fig:G-Gamma-code} in \ref{sec:app-code}
    
    Let us explore several illustrative examples.
    
    \begin{Example}\label{eg:K_n}
    Suppose that $\Gamma$ is a complete graph $K_n$. Then since every vertex is adjacent to every other vertex, clicking any vertex with $g$ has the effect of multiplying every coordinate by $g$. Thus $G^\Gamma$ is the \emph{diagonal subgroup} of $G^n$, where every coordinate is equal. This shows that the group $G^\G$ is a generalization of the diagonal subgroup of the product of groups.
    \end{Example}

    \begin{Example}
    Imagine a graph $\G$ with a Rubik's Cube attached to each vertex, and any time we perform a move on one cube, that same move is performed on its neighbors. Then we can ask questions like: is it possible to simultaneously solve every Rubik's cube this way, no matter what the original configuration is? This is equivalent to the question of whether $G^\G = G^{|\G|}$.
    \end{Example}

    \begin{Example}
    More generally, suppose that we attach a set $S_i$ to each vertex $v_i$ of $\G$ and that $G$ acts on each $S_i$. We model the state of this ``game'' with an element of $S = S_1 \times \cdots \times S_{|\G|}$, and clicking a vertex with $g$ acts on the state of that vertex and its neighbors by $g$. This gives us an action of $G^\G$ on $S$ which generalizes the \emph{parallel product} of group actions (see \cite{parallel-product}); indeed, if $\G = K_n$, then the action of $G^\G$ is the parallel product of the actions on the $n$ sets.
    \end{Example}
    
    \begin{Example} 
    Suppose that $\Gamma$ is a path graph of length 3, with vertices numbered in the natural way. Then clicking vertex 2 with $g$ and clicking vertex 1 with $g^{-1}$ yields $(1,1,g)$, so $(1,1,g) \in G^\Gamma$ for every $g$. Similarly, $(g,1,1) \in G^\Gamma$. Finally, since clicking vertex 2 with $g$ yields  $(g,g,g)$, it follows that $(1,g,1) \in G^\Gamma$, and thus $G^\Gamma = G^3$. This means it is indeed always possible to solve three Rubik's cubes that are connected in a row as described in the introduction, provided each one is in a solvable state itself (let $G$ = the Rubik's Cube group; we have shown that we can act on each Rubik's Cube individually by using combinations of moves).
    \end{Example}
    
    \begin{Example}
    Now suppose that $\Gamma$ is a cycle with $4$ vertices, and let $G = (\Z, +)$, the additive group of the integers. Then $G^\Gamma$ is the subgroup of the additive group of $\Z^4$ that is generated by the four permutations of $(1,1,1,0)$. Clearly, every element of $G^\Gamma$ has a coordinate sum that is divisible by 3, so it is a proper subgroup of $\Z^4$. 

    If instead, we set $G = \Z_2$, then clicking any three vertices with $1$ yields a permutation of $(1, 0, 0, 0)$. All four permutations are obtainable this way, and thus, $G^\Gamma = G^4$.
    \end{Example}
    
    What we see here is that for some graphs, $G^\Gamma$ is the whole direct product $G^n$ regardless of $G$. Similarly, sometimes $G^\Gamma$ is a proper subgroup of $G^n$ regardless of $G$. On the other hand, for some graphs, the choice of $G$ influences the index of $G^\Gamma$ in $G^n$.
    
    Note that if $G = \Z_2$, then $G^\Gamma$ is essentially the state space of playing ``Lights Out'' on $\Gamma$. In the literature \cite{LOFiniteGraphs}, a graph $\Gamma$ is Always Winnable (AW) if, essentially, $\Z_2^\Gamma = \Z_2^n$. 

    More generally, we will see that $\Z$ plays an important role. The following result follows from the definition of $G^S$, but we will find it useful to spell things out. 
    
    \begin{prop} \label{Z-AW-universal}
    If $S \subseteq \Z^n$ and $\langle S \rangle = \Z^n$ (so that $\Z^S = \Z^n$), then $G^S = G^n$ for every group $G$. In particular, if $\Z^\G = \Z^{|\G|}$, then $G^\G = G^{|\G|}$ for every group $G$.
    \end{prop}
    \begin{proof}
    If $\langle S \rangle = \Z^n$, then in particular, $\bbe_i \in \langle S \rangle$ for each $i = 1, \ldots, n$. In other words, if $S = \{\bx_1, \ldots, \bx_m\}$, there is a linear combination that yields $\bbe_i$, say
    \[ a_1 \bx_1 + \cdots + a_m \bx_m = \bbe_i, \]
    for some integers $a_j$. Now, for every $g \in G$, it follows that $g^{a_1 \bx_1} \cdots g^{a_m \bx_m} = g^{\bbe_i}$. In other words, $G^S$ contains the element of $G^n$ with $g$ in coordinate $i$ and the identity everywhere else. Then clearly, we can set each coordinate of an element of $G^S$ independently from the rest, and so $G^S = G^n$.
    \end{proof}
    
    

    The argument of Proposition~\ref{Z-AW-universal} can be generalized as follows; the argument is essentially the same.

    \begin{prop} \label{kZ-pseudo-universal}
    Let $k > 1$ be an integer. Suppose $S \subseteq \Z^n$ and that $(k \Z)^n \leq \langle S \rangle$. Fix a group $G$ and let $H$ be the subgroup of $G$ generated by $k$-th powers. Then $H^{n} \leq G^S$. In particular, if the $k$th powers of elements of $G$ generate $G$, then $G^S = G^{n}$.
    \end{prop}

    We see that the structure of $\Z^\Gamma$ is instrumental in understanding $G^\Gamma$. We will focus on determining $\Z^\Gamma$ in \cref{sec:abelian}.
        
\subsection{Subgroups of $G$ and $G^S$}
    We will now work out some basic structural results of $G^\Gamma$ and, more generally, $G^S$. 
    Recall that the commutator subgroup $[G, G]$ generated by all the commutators $[x, y] \in G$ is a normal subgroup; we will consider arbitrary normal subgroups $K$ here for generality, but the reader may find it convenient to keep in mind that $K$ will be $[G, G]$ in the most important applications in the future. We will need to think about the abelianization $G/[G, G]$ of $G$, and then the decomposition of that abelianization as the direct product of cyclic groups, so we start with two results that will be instrumental in the rest of the paper:

    \begin{prop} \label{prop:simpleprops}
    Let $S \subseteq \Z^n$.
    \begin{enumerate}[(a)]
        \item If $K \lhd G$, then there is a natural epimorphism from $G^S$ to $(G/K)^S$ with kernel $K^n \cap G^S$.
        \item $(G \times H)^S$ is naturally isomorphic to $G^S \times H^S$.
    \end{enumerate}
    \end{prop}

    \begin{proof}
    The first result is straightforward. For the second one, let $K \lhd G$ and let $\varphi: G^n \to (G/K)^n$ be the natural epimorphism. This induces a natural epimorphism from $G^S$ to $(G/K)^S$, and the kernel consists of all elements of $G^S$ that have an element of $K$ in each component.
    \end{proof}

    \begin{prop} \label{AW-projects}
    Let $S \subseteq \Z^n$. 
    If $K \lhd G$ and $G^S = G^{n}$, then $(G/K)^S = (G/K)^{n}$.
    \end{prop}
    
    \begin{proof}
    For any element $h$ of $(G/K)^{n}$, pick a preimage $g$ in $G^{n}$. Then we can write $g$ as a product of generators of $G^S$, and the product of the images of these generators in $(G/K)^S$ yields $h$.
    \end{proof}
    
    \begin{prop} \label{AW-splitting}
    Let $S \subseteq \Z^n$ and let $K \lhd G$. Then $G^S = G^{n}$ if and only if $K^{n} \leq G^S$ and $(G/K)^S = (G/K)^{n}$. \end{prop}
    
    \begin{proof}
        If $G^S = G^{n}$, then clearly $K^{n} \leq G^S$. Proposition~\ref{AW-projects} proves the rest of that direction.
    
        Now suppose that $K^{n} \leq G^S$ and that $(G/K)^S = (G/K)^{n}$. Consider $(g_1, \ldots, g_n) \in G^n$. Projecting to $(G/K)^n$ we get $(K g_1, \ldots, K g_n)$ in $(G/K)^n$. Since $(G/K)^S = (G/K)^n$, we can write this element as $h_1^{\bx_1} \cdots h_m^{\bx_m}$ in $(G/K)^S$. That is, each $h_i \in G/K$. For each $h_i$, let $\tilde{h_i} \in G$ such that $\tilde{h_i}$ projects to $h_i$. Then $\tilde{h_1}^{\bx_1} \cdots \tilde{h_m}^{\bx_m}$ is an element of $G^S$ that projects to $h_1^{\bx_1} \cdots h_m^{\bx_m} = (K g_1, \ldots, K g_n)$. Thus, the element of $G^S$ differs from $(g_1, \ldots, g_n)$ only by an element of $K^n$, and since $K^n \leq G^S$, it follows that $(g_1, \ldots, g_n) \in G^S$ as well.
    \end{proof}
   
    \begin{coro} \label{AW-splitting-2}
    Let $S \subseteq \Z^n$ and let $K \lhd G$. Suppose that $K^S = K^{n}$ and $(G/K)^S = (G/K)^{n}$. Then $G^S = G^{n}$.
    \end{coro}
    
    \begin{remark}
    Note that the converse of Corollary~\ref{AW-splitting-2} is not true, even in the restricted case where we work with $G^\Gamma$ in place of $G^S$. For example, consider $\Gamma = C_5$, $G = S_4$, and $K = [G, G] = A_4$. In GAP we verified that $(S_4)^{C_5} = (S_4)^5$, but $(A_4)^{C_5} \neq (A_4)^5$. Thus in this case, we have $K^5 \leq G^\Gamma$, but $K^5 \not \leq K^\Gamma$. In other words, we need to click with elements outside of $K$ in order to get some elements of $K^5$.
    \end{remark}

\subsection{How the structure of $\Gamma$ affects $G^\Gamma$}

    Our main interest is in the group $G^\Gamma$ for graphs $\Gamma$. Let us study how some basic graph properties of $\Gamma$ affect $G^\Gamma$. The first result is simple.
    
    \begin{prop} \label{prop:conn-comp}
    If $\Gamma$ has connected components $\G_1, \ldots, \G_k$, then $G^\Gamma = G^{\G_1} \times \cdots \times G^{\G_k}$.
    \end{prop}
    
    Next we extend Example~\ref{eg:K_n}:
    
    \begin{defi}
    Two vertices $i$ and $j$ of $\Gamma$ with $i \neq j$ are \emph{neighborhood-indistinguishable} if $\nbhd{i} = \nbhd{j}$; otherwise they are \emph{neighborhood-distinguishable}.
    \end{defi}
    
    \begin{prop} \label{prop:indistinguishable}
    Suppose that $\Gamma$ is a graph with vertices $i$ and $j$ that are neighborhood-indistinguishable. Let $\Gamma - j$ be the subgraph of $\Gamma$ induced by all vertices except for $j$. Then $G^\Gamma \cong G^{\Gamma - j}$.
    \end{prop}
    
    \begin{proof}
    First, note that if $i$ and $j$ are neighborhood-indistinguishable, then $\pi_i(x) = \pi_j(x)$ for all $x \in G^\Gamma$, since every element $g_{B(k)}$ that affects coordinate $i$ also affects coordinate $j$ in the same way.
    
    There is a natural epimorphism from $G^\Gamma$ to $G^{\Gamma - j}$ where we just forget coordinate $j$. The kernel consists of all elements of $G^\Gamma$ of the form $g_j$ (that is, with only the identity in every coordinate other than $j$). Since $\pi_i(x) = \pi_j(x)$ for all $x \in G^\Gamma$, the only element of $G^\Gamma$ of the form $g_j$ is the identity.
    \end{proof}
    
    In light of Propositions~\ref{prop:conn-comp} and \ref{prop:indistinguishable}, henceforth we will assume that our graphs are connected and that all vertices are neighborhood-distinguishable. The number of such graphs on $n$ vertices is given by \cite[A004108]{oeis}; this is not obvious from the definition, but is established in \cite{mating-graphs}, which studies ``mating graphs'' where no two vertices have the same closed neighborhood.

\section{Graph powers of an abelian group} \label{sec:abelian}

\subsection{Determining $G^S$ from the lattice $\mathcal{L} = \langle S\rangle$}

We saw in Propositions~\ref{Z-AW-universal} and \ref{kZ-pseudo-universal} some ways that the structure of $\langle S \rangle$ can affect the structure of $G^S$. In fact, we will see that if $G$ is abelian, then in some sense $G^S$ is completely determined by $\langle S \rangle$. 

We will find it easier to work with matrices here rather than subsets of $\Z^n$. If $S \subseteq \Z^n$ and $|S| = m$, then we can find a basis for the integer lattice $\mathcal{L} = \langle S \rangle$ and then construct an $m \times n$ matrix $M$ whose rows are the vectors in the chosen basis. Then $G^M = G^{\langle S \rangle} = G^S$. 

If $G$ is abelian, one way of understanding the structure of $G^M$ is by working over the principal ideal domain $\Z$ and finding the Smith Normal Form $N$ of $M$ (cf. \cite{DummitandFoote}) and the elementary divisors. If we have $U \in GL_m(\Z), V \in GL_n(\Z)$ so that $UMV = N$, then $G^S = G^nM = G^nU^{-1}NV^{-1} = G^nNV^{-1}$ is isomorphic to $G^nN$. If the main diagonal of $N$ contains $s$ 1's, followed by $a_1, a_2, \dots, a_t$ (with $a_1 \mid a_2 \mid \cdots \mid a_t$) and then $u$ 0's (where $u = \min(m, n) - s - t$), then we have 
\begin{align}
    \label{desc-of-G^S-from-eldivs}
    &G^M \cong G^s \oplus \bigoplus_{i=1}^t a_i G\qquad \text{ and }\\
    \label{desc-of-kernel}
    &G^n / G^M \cong \bigoplus_{i=1}^t G/a_iG \oplus G^u
\end{align}
Thus, below we will content ourselves with finding the $s + t$ nonzero elementary divisors of the matrix $M$ and allow the reader to use Equations \eqref{desc-of-G^S-from-eldivs} and \eqref{desc-of-kernel} to describe $G^S$. See \cite{electrician}, Theorem~5 for one explicit description of the tuples in $G^S$ when $G$ is finite.

\subsection{Description of $\Z^\G$ for some basic graphs}
\label{sect:ZG-basic-graphs}
For the remainder of this section, we will work with the special case $M = A_\G$ for some graph $\G$ and determine the elementary divisors for the activation matrices corresponding to some well-known families of graphs. From there, Equations \eqref{desc-of-G^S-from-eldivs} and \eqref{desc-of-kernel} can be used to describe $G^\G$ for every abelian group $G$.

We will prove the following results by finding a convenient basis for $G^\G$ that essentially starts with rows with a leading 1, meaning we can set the value of the corresponding vertex to anything we like, regardless of the values of the vertices that have come before. We will then show the rest of the vertices are completely fixed (meaning their corresponding row can be removed from $A_\G$ without any loss since that row can be expressed as a linear combination of previous rows) or can be changed by a multiple of some number $x$ (meaning $x$ is the leading entry in the row).

\begin{prop} \label{prop:Z-path}
Let $\G = P_n$, the path graph with $n$ vertices. If $n \equiv 2$ (mod $3$), then the elementary divisors of $\G$ are $(1^{n-1}, 0)$, and otherwise they are $(1^n)$ so that $\Z^\G = \Z^n$.
\end{prop}
\begin{proof}
It is always possible to set the first $n-1$ coordinates to be anything we wish (say $(a_1, \ldots, a_{n-1}, *)$); first click $v_2$ with $a_1$, and then proceed to the right, always clicking $v_i$ with whatever we need to set the $(i-1)$-st coordinate to $a_{i-1}$. The only question that remains is how much freedom we have to set the last coordinate. Let $\bx_1, \ldots, \bx_n$ be the rows of $A_\G$. If $n \equiv 0$ (mod $3$), then
\[ \bbe_n = (\bx_2 + \bx_5 + \cdots + \bx_{n-1}) - (\bx_1 + \bx_4 + \cdots + \bx_{n-2}), \]
and so we can also set the last coordinate arbitrarily. Similarly, if $n \equiv 1$ (mod $3$), then
\[ \bbe_n = (\bx_1 + \bx_4 + \cdots + \bx_n) - (\bx_2 + \bx_5 + \cdots + \bx_{n-2}), \]
and so we can also set the last coordinate arbitrarily in that case. Thus, if $n \not \equiv 2$ (mod $3$), then $\Z^\G = \Z^n$.


Now suppose that $n \equiv 2$ (mod $3$). Then
\[ \bx_{3n+2} = (\bx_1 + \bx_4 + \cdots + \bx_{3n+1}) - (\bx_2 + \bx_5 + \cdots + \bx_{3n-1}), \]
which means the last row of an echelon form of $M$ must be zero, hence so must the last elementary divisor.
\end{proof}

\begin{prop} \label{prop:Z-cycle}
Let $\G = C_n$, the cycle graph with $n$ vertices. If $n$ is divisible by $3$, then the elementary divisors are $(1^{n-2}, 0^2)$, and otherwise they are $(1^{n-1}, 3)$.
\end{prop}
\begin{proof}
Let $\bx_i = \bbe_{i-1} + \bbe_i + \bbe_{i+1}$ with indices mod $n$. We can set the first $n-2$ coordinates to be anything, say $a_1, \ldots, a_{n-2}$: click $v_2$ with $a_1$, and work our way to the right, clicking $v_{i+1}$ with whatever we need to make the $i$th coordinate equal to $a_i$. This shows that the first $n-2$ elementary divisors are 1. The question that remains is which vectors of the form $(0, \ldots, 0, *, *)$ we can obtain.

First, suppose that $n = 3k$. Then
\[ \bx_{3n} = (\bx_1 + \bx_4 + \cdots + \bx_{3k-2}) - (\bx_3 + \bx_6 + \cdots + \bx_{3k-3}) \]
and
\[ \bx_{3n-1} = (\bx_1 + \bx_4 + \cdots + \bx_{3k-2}) - (\bx_2 + \bx_5 + \cdots + \bx_{3k-4}). \]
This implies the last two rows of $H_\G$ are zero.

Now suppose that $n$ is not divisible by 3. If $n = 3k+2$, then
\begin{align*}
    (\bx_1 + \bx_4 + \cdots + \bx_{3k+1}) &- (\bx_2 + \bx_5 + \cdots + \bx_{3k-1})\\
    &= (0, 0, \ldots, 0, 1, *).
\end{align*}
Similarly, if $n = 3k+1$, then
\begin{align*}
    -(\bx_1 + \bx_4 + \cdots + \bx_{3k-2}) &+ (\bx_2 + \bx_5 + \cdots + \bx_{3k-1})\\
    &= (0, 0, \ldots, 0, 1, *).
\end{align*}
This gives us an additional elementary divisor of 1.

It remains to find the last elementary divisor. Note that
\[ \sum \bx_i = (3, 3, \ldots, 3), \]
and since $\bbe_1, \ldots, \bbe_{n-2}$ are in $\Z^\G$, it follows that $3 \bbe_n \in \Z^\G$. On the other hand, since every row has sum 3, Proposition~\ref{row-sums-fund-factors} implies that the last elementary divisor is divisible by 3. It follows that the last elementary divisor is precisely 3.
\end{proof}

\begin{prop} \label{prop:Z-bipartite}
Let $\Gamma = K_{m,n}$, the complete bipartite graph. Then the elementary divisors of $\G$ are $(1^{m+n-1}, mn-1)$. 
\end{prop}
\begin{proof}
Let us label the vertices in one partition $u_1, \ldots, u_m$ and in the other $v_1, \ldots, v_n$. First, let us show that if we drop $v_n$, we can get any combination of integers in the remaining coordinates. Suppose we want to end up with $x_i$ in each $u_i$ and $y_j$ in each $v_j$ except for $v_n$. Let $x = \sum x_i$. If we click each $u_i$ with $x_i$, click $v_j$ with $-x + y_j$ for $j \in \{1, \ldots, n-1\}$, and then click $v_n$ with $(n-1)x - (y_1 + \cdots + y_{n-1})$, then we accomplish our goal.

Before proceeding, let us prove a lemma. Suppose $v_i$ and $v_j$ in $\Gamma$ have the same neighbors. If $a_1 \bx_1 + \cdots + a_n \bx_n$ has coordinates $i$ and $j$ equal, then $a_i = a_j$. To see this, let $\bv = a_1 \bx_1 + \cdots + a_n \bx_n$ and let $s = \sum a_k$, over those $k$ such that $v_k$ is a neighbor of $v_i$. Then the value of coordinate $i$ in $\bv$ is $s + a_i$, and it is $s + a_j$ in coordinate $j$, and the lemma follows.

Now consider the elements of $\Z^\Gamma$ that have a 0 in every coordinate other than $v_n$. By the lemma above, the vertices $u_i$ are all clicked with the same element, say $a$, and the vertices $v_1, \ldots, v_{n-1}$ are all clicked with the same element, say $b$. Let $c$ be the element that we click $v_n$ with. Now, considering that we want the state of $v_1$ to be $0$ means that $b + ma = 0$, and so $b = -ma$. Then, considering that we want the state of $u_1$ to be $0$ means that $a + (n-1)b + c = 0$. Thus,
\[ a + (n-1)(-ma) + c = 0, \]
and so $c = a((n-1)m - 1)$. Then the state of vertex $v_n$ is 
\[ am + c = am + a((n-1)m - 1) = (mn-1)a. \]
Since $a$ was arbitrary to begin with, we have described a basis for $G^\G$ consisting of rows of the form $(0, \dots, 0, 1, *, \dots, *)$ and $(0, \dots, 0, mn-1)$. Thus we see that the last elementary divisor is $(mn-1)$ .
\end{proof}

\begin{coro}
Let $\Gamma = \St_n$, the star graph with $n+1$ vertices. Then the elementary divisors of $\G$ are $(1^n, n-1)$. 
\end{coro}

\subsection{The elementary divisors of small graphs}

Using the computational ideas in the previous section, along with the database of all graphs with up to 7 vertices in SageMath, we determined that the number of graphs on $n$ vertices such that $\Z^\Gamma = \Z^{|\Gamma|}$ is 
\[ 1, 0, 1, 1, 6, 20, 172, \ldots. \]

\cref{tab:small-non-AW-graphs} lists the elementary divisors for those graphs on up to 5 vertices with the property that $\Z^\Gamma \neq \Z^{|\Gamma|}$:

\begin{table}[htbp]
    \centering
\begin{tabular}{|c|c|c|} \hline
\textbf{$n$} & \textbf{Graph} & \textbf{El. div.'s} \\ \hline
4 & $C_4$ & $(1^3, 3)$ \\ \hline
4 & $\St_3$ & $(1^3, 2)$ \\ \hline
5 & $P_5$ & $(1^4, 0)$ \\ \hline
5 & $C_5$ & $(1^4, 3)$ \\ \hline
5 & $\St_4$ & $(1^4, 3)$ \\ \hline
5 & $K_{2,3}$ & $(1^4, 5)$ \\ \hline
5 & tadpole graph $T_{4,1}$ & $(1^4, 2)$ \\ \hline
\end{tabular}
    \caption{Graphs on up to 5 vertices with $\Z^\G \neq \Z^{|\G|}$}
    \label{tab:small-non-AW-graphs}
\end{table}

Let us highlight one more graph here, because it will be a prominent example later.

\begin{prop}
\label{prop:Q3_ZZ}
The elementary divisors of $\G = Q_3$, the graph of the $3$-cube, are $(1^4, 2, 0^3)$.
\end{prop}

\begin{proof}
This is easily checked computationally with SageMath.

Alternatively, it's clear that we can set 4 coordinates (for example, all "top" vertices) to be anything we want by only clicking the bottom vertices. Then some slightly tedious algebra shows that the kernel of the projection onto those 4 coordinates has the other 4 coordinates all equal to any particular even integer.
\end{proof}

\section{Reducing nonabelian graph powers to the abelian case}
\label{nonabelian-g-gamma}

    \subsection{Definitions}
    \label{defs}

    Having seen how to determine the structure of $G^\Gamma$ when $G$ is abelian, we now turn to the case of nonabelian groups. We will exploit Proposition~\ref{prop:simpleprops}(a), which implies that there is an epimorphism from $G^\G$ to $(\Gab)^\G$ with kernel $[G,G]^{|\G|} \cap G^\G$. If we can determine $[G,G]^{|\G|} \cap G^\G$, then combined with the results from the previous section, this will give us a description of $G^\G$. We will define things in the general context of $G^S$, though the fact that the rows of $A_\G$ consist only of $0$s and $1$s is an essential component of much of our analysis.

    \begin{defi}
    If $S \subseteq \Z^n$, then the \emph{coordinate-wise commutator subgroup} $\CComm(G, S)$ is $[G,G]^{n} \cap G^S$. In other words, it is those elements of $G^S$ with an element of $[G,G]$ in each coordinate. If $\G$ is a graph, $\CComm(G, \G)$ is $[G,G]^{|\G|} \cap G^\G$.
    \end{defi}
    

    We will see that for a wide-ranging class of graphs, $[G,G]^{|\G|} \leq G^\G$ (so that $\CComm(G, \G) = [G,G]^{|\G|}$), which makes the description of $G^\G$ depend in a straightforward way on $(G^\Ab)^\G$. We thus make the following definition.

    \begin{defi}
    The graph $\Gamma$ is \emph{$G$-RA} (short for \emph{reducible to abelian}) if $[G,G]^{|\Gamma|} \le G^\G$.
    If $\Gamma$ is $G$-RA for every group $G$, then we simply say that $\Gamma$ is RA.
    \end{defi}

    \begin{Example}\label{eg:Rubik_full}
        It is well-known that the commutator subgroup of the Rubik's Cube group $G$ has index 2 ($[G, G]$ is just the subgroup of even permutations of the faces of the cubies), which implies the following.
        
        Consider again the ``Rubik's Cube game'' on $\G$, where each vertex has a Rubik's Cube attached and a move at one vertex affects the cube at that vertex and those at the adjacent vertices. Let $G$ be the Rubik's Cube group, and suppose that $\G$ is $G$-RA. 
        
        For each  Rubik's cube at every vertex, we can determine whether the starting position is in $[G, G]$ or not by determining if it is an odd or even permutation from the solved state. Then we can write down the state as a 1 if odd or 0 if even. The overall starting position of the Rubik's Cubes will be solvable if and only if the corresponding Lights Out Game on the graph has a solution (that is, if you can get back to the all-0 state). This is because a solution exists if and only if we can perform moves on the Rubik's cubes to make each state reduce to $\bar 0$ in the abelianization $G/[G, G]$ in each coordinate. We know we can act on each coordinate by any element of $[G, G]$ independently, so we can get to the solved state independently for each cube.

        As a corollary, every initial state is solvable if and only if every initial state of the Lights Out game on $\G$ is solvable.
    \end{Example}  

    In general, determining $\CComm(G,S)$ directly is difficult. We will make use of several subgroups of $[G,G]^{n}$. First, note that since $(\Gab)^S$ is abelian, that implies that $[G^S, G^S]$ is contained in $\CComm(G, S)$. Further, let us define two subgroups of $[G^S, G^S]$.

    \begin{defi}
        The \emph{basic commutator subgroup} $\Comm(G, S)$ is the subgroup of $[G^S, G^S]$ generated by $[g^{\bx}, h^{\by}]$ for any $\bx, \by \in S$ and elements $g, h \in G$. Similarly, if $\G$ is a graph, then $\Comm(G, \G)$ is the subgroup of $[G^\G, G^\G]$ generated by $[g^u, h^v]$ for any vertices $u, v \in \G$ and elements $g, h \in G$.

        The \emph{distinct basic commutator subgroup} $\Commd(G, S)$ is the subgroup of $[G^S, G^S]$ generated by $[g^\bx, h^\by]$ for any $g, h \in G$ and \emph{distinct} vectors $\bx \ne \by \in S$. Similarly, $\Commd(G, \G)$ is the subgroup of $[G^\G, G^\G]$ generated by $[g^u, h^v]$ for any $g, h \in G$ and \emph{distinct} vertices $u \ne v \in \G$. 
    \end{defi}

    Thus, in general we have
    \begin{align}
        \label{all-layers}
        \Commd(G, S) &\le \Comm(G, S) \le [G^S, G^S]\\
        \nonumber & \le \CComm(G, S) \le [G, G]^{n}
    \end{align}
    Note that $[G^S, G^S]$ is the normal closure of $\Comm(G, S)$ in $G^S$, and since $[G, G]^S$ is the subgroup of $G^S$ generated by $[g^\bx, h^\by]$, we have $\Comm(G, S) = \langle \Commd(G,S), [G, G]^S\rangle$.

    For each of the containments in \eqref{all-layers}, there are examples that show that the containment may be proper. Proposition~\ref{even-cube-RA} below shows that for even-dimensional cubes $\G = Q_d$, $\Commd(D_8, \G)$ does not equal $\Comm(D_8, \G)$. The examples in \cref{SuffRAMatrix} show $\Comm(G, S)$ may not equal $[G^S, G^S]$ and $[G^S, G^S]$ may not equal $[G, G]^{|S|} \cap G^S$. Furthermore, in a future paper we will show that the crown graph on $2p+4$ vertices (i.e. the tensor product $K_2 \times K_{p+2}$) is an example of a graph where $[G^\G, G^\G] \ne \CComm(G, \G)$ for $G = H(\F_p)$ (see \cite{FullPaper}). Finally, as we will show in Corollary~\ref{graphs-not-RA}, the cube graphs $Q_{2n+1}$ of odd dimension are not RA. 
    
    For the rest of this section, we will shift to working with $G^\G$ only. We will show that many graphs have the property that for every $G$, the group $\Commd(G, \G)$ or $\Comm(G, \G)$ is equal to all of $[G,G]^{|\G|}$. 
    In particular, when $\Commd(G,\G) = [G,G]^{|\G|}$, the groups in \eqref{all-layers} all coincide and $\G$ is RA.   
    The group $\Commd(G, \G)$ is also instrumental in analyzing $G^\G$ when $\G$ is a product of graphs (see \cite{FullPaper}), so we make the following definition.

    \begin{defi}
    The graph $\G$ is \emph{strongly RA} if, for every group $G$,
    $\Commd(G, \G) = [G,G]^{|\G|}$. In particular, if $\G$ is strongly RA, then it is RA.
    \end{defi}

    In order to understand the groups $\Commd(G, \G)$ and $\Comm(G, \G)$, let us analyze the commutators $[g^u, h^v]$.

    \begin{prop} \label{comm-char}
    Let $u, v \in \Gamma$, and let $\bx$ be the indicator vector for $B(u) \cap B(v)$. Then $[g^u, h^v] = [g,h]^\bx$.
    \end{prop}

    \begin{proof}
    It suffices to consider the effect that $[g^u, h^v]$ has on three different kinds of vertices:
    \begin{enumerate}
        \item If $w$ is not adjacent to $u$ or $v$, then both $g^u$ and $h^v$ leave it unchanged.
        \item If $w$ is adjacent to $u$ but not $v$, then $w$ is affected only by $g^u$ and $(g^{-1})^u$, and the net effect is that it is unchanged. Similarly, if $w$ is adjacent to $v$ but not $u$, it ends up unchanged.
        \item If $w$ is adjacent to $u$ and $v$, then the state of $w$ is multiplied by $[g,h]$.
    \end{enumerate}
    Thus $[g^u, h^v]$ precisely gives $[g, h]$ at the vertices in $B(u) \cap B(v)$.
    \end{proof}

    Let us consider these subgroups in a simple graph and show how it can help establish that a graph is RA.

    \begin{Example} \label{eg:RA-graph}
    Let $\Gamma$ be the cycle graph $C_4$. Label the vertices cyclically $v_1, \ldots, v_4$. Then
    \[ [g^{v_1}, h^{v_{2}}] = ([g,h], [g,h], 1, 1), \]
    \[ [g^{v_1}, h^{v_3}] = (1, [g,h], 1, [g,h]). \]
    The group $\Commd(G, C_4)$ is generated by elements like these; namely, elements of $G^{C_4}$ that have the same commutator on any pair of vertices. If $G = D_8$, then there is only a single nontrivial commutator $c$, and it has order $2$. It follows that $\Commd(D_8, C_4)$ is the subgroup of $\langle c \rangle^4$ that has $c$ on an even number of coordinates, a group of order 8.

    In the group $\Comm(G, C_4)$, we are additionally allowed to use
    \[ [g^{v_1}, h^{v_1}] = ([g,h], [g,h], 1, [g,h]). \]
    Then it follows that $\Comm(G, C_4)$ contains
    \[ [g^{v_1}, h^{v_1}] [h^{v_1}, g^{v_3}] = ([g,h], 1, 1, 1). \]
    Similar arguments show that $[g,h]^{\bbe_i} \in G^{C_4}$ for all $i$. It follows that $\Comm(G, C_4) = [G,G]^{4}$ and so $C_4$ is RA (but not strongly RA). 
    \end{Example}

    In some graphs, the relationship between $[G, G]^{|\G|}$ and $G^\G$ depends on $G$:
    
    \begin{Example} \label{eg:not-RA}
    Let $\Gamma$ be the graph of a cube. Using GAP, we find that if $G$ is dihedral of order $10$, then indeed $[G, G]^8 \leq G^\G$. On the other hand, if $G$ is dihedral of order $8$, then $\CComm(G, \G)$ has index $2$ in $[G,G]^8$. See Figure~\ref{fig:Q3-RA-code} in \ref{sec:app-code}
    
    More generally, suppose that the vertex $v_1$ in the cube has neighbors $v_2$, $v_3$, and $v_4$. Then:
    \[ [h^{v_1}, g^{v_1}] \cdot [g^{v_1}, h^{v_2}] \cdot [g^{v_1}, h^{v_3}] \cdot [g^{v_1}, h^{v_4}] = ([g,h]^2)^{e_1}. \]
    Similar computations show that for any commutator $c$ and any $i = 1, 2, \dots, 8$, we have $c^{2\bbe_i} \in [G^\Gamma, G^\Gamma]$. Thus, if $[G,G]$ is generated by the squares of commutators, then $[G, G]^8$ is contained in $G^\G$. 
    \end{Example}

\subsection{Graph families that are RA}
\label{RA-families}

    For a fixed group $G$ and graph $\G$, there are computational methods to help determine whether $\G$ is $G$-RA. We defer these methods to \cref{compute-RA}. For now, our goal is to find ways to show that a graph $\G$ is RA without extensive calculation. Typically, this only involves working with elements of the form $[g^u, h^v]$ (i.e. with $\Comm(G,\G)$ or $\Commd(G,\G)$).

    The first result follows directly from Proposition~\ref{comm-char}.

    \begin{prop} \label{adj-comms}
        If $u$ and $v$ are adjacent vertices with no neighbors in common, then $[g,h]^{\bbe_u+\bbe_v} \in \Commd(G, \G)$. 
    \end{prop}
    
    \begin{lemma} \label{corner-trick}
    If $u$ is a vertex of degree two, with its neighbors $v$ and $w$ not connected by an edge, then $[g, h]^{\bx} \in \Comm(G, \G)$ for $\bx = \bbe_u, \bbe_v, \bbe_w$ and any $g, h \in G$.
    \end{lemma}

    \begin{proof}
    By Proposition~\ref{adj-comms}, $[g,h]^{\bbe_u+\bbe_v} \in \Commd(G, \G)$. Multiplying by $[h,g]^{u}$ gives us $[h,g]^{\bbe_w} \in \Comm(G, \Gamma)$. By a similar argument, $[h,g]^{\bbe_v} \in \Comm(G, \Gamma)$, and multiplying $[h,g]^{\bbe_v} [h,g]^{\bbe_w}$ by $[g, h]^{u}$  implies that $[g,h]^{\bbe_u} \in \Comm(G, \Gamma)$, proving the result.
    \end{proof}
    
\begin{lemma} \label{threeinarow}
    If there exists a vertex $u \in \Gamma$ with neighbors $v$ and $w$ that have no common neighbors other than $u$, then $[g, h]^{\bbe_u} \in \Commd(G ,\G)$ for every $g, h \in G$.
\end{lemma}
\begin{proof}
    Under these assumptions, $[g^{v}, h^{w}] = [g,h]^{\bbe_u}$.
\end{proof}

Having established several `local' results, we now have the tools we need to prove some `global' results.

\begin{lemma} \label{tri-free-all-comms}
    If $\Gamma$ is a connected triangle-free graph and there exists a vertex $v$ such that $[g, h]^{\bbe_v}$ is in $ \Commd(G, \G)$ (resp. $\Comm(G, \G), [G^\G, G^\G]$) then $[G, G]^{|\G|}$ equals $\Commd(G, \G)$(resp. $\Comm(G, \G),  [G^\G, G^\G]$) . In any case, $\Gamma$ is RA. That is, if we can put any commutator in a single vertex, then we can put any commutator by itself anywhere.
\end{lemma}
\begin{proof}
    Suppose $[g,h]^{\bbe_v} \in G^{\Gamma}$. The fact that the graph is triangle-free implies than any neighbor $w$ of $v$ shares no neighbors with $v$, so by Proposition~\ref{adj-comms}, we can multiply by $[h, g]^{\bbe_v + \bbe_w}$, and so $[h, g]^{\bbe_w} \in G^\Gamma$ and thus $[g,h]^{\bbe_w} \in G^\Gamma$. Similarly, we can extend the reasoning to neighbors of neighbors and so on, and since $\G$ is connected, it follows by induction that $[g,h]^{\bbe_u} \in G^\G$ for every vertex $u$. Furthermore, since we obtained $[g,h]^{\bbe_u}$ by multiplying $[g,h]^{\bbe_v}$ by elements of $\Commd(G,\G)$, it follows that $[g,h]^{\bbe_u}$ is in $\Commd(G,\G)$ (resp. $\Comm(G,\G)$) if $[g,h]^{\bbe_v}$ is in $\Commd(G,\G)$ (resp. $\Comm(G,\G)$). 
\end{proof}

Recall that the \emph{girth} of a graph is the length of its smallest cycle.

\begin{coro} \label{cor:girth-5}
    If $\Gamma$ is a connected graph with girth $5$ or more and at least $3$ vertices, then $\Gamma$ is strongly RA. 
\end{coro}

\begin{proof}
    Let $\Gamma$ be a graph satisfying the hypotheses. Then there is some vertex $u$ with degree 2 or more. Consider neighbors $v$ and $w$ of $u$. Since $\Gamma$ has girth at least 5, $v$ and $w$ are not adjacent, and they have no neighbors in common other than $u$. Thus $[g,h]^{\bbe_u}$ is in $\Commd(G, \Gamma)$ by Lemma~\ref{threeinarow}, and Lemma~\ref{tri-free-all-comms} completes the proof.
\end{proof}

Recall for a moment that we restrict ourselves to connected graphs that are neighborhood-distinguishable; that is, no vertices $u$ and $v$ have the same closed neighborhood. Note that if $|\G| \geq 3$ and $u$ and $v$ have the same closed neighborhood, then they have a common neighbor $w$ and this gives us a triangle in our graph. Thus, as long as $\G$ is triangle-free, we do not have to worry about checking that the graph is neighborhood-distinguishable. 

Let us now use Corollary~\ref{cor:girth-5} and highlight a few graphs of interest.

\begin{coro} \label{girth-5-ra-ex}
    The following graphs are all strongly RA:
    \begin{enumerate}
        \item The Petersen graph
        \item $C_n$ with $n \geq 5$
        \item Trees with at least $3$ vertices.
    \end{enumerate}
\end{coro}

Having fully determined $\CComm(G, \G)$ for graphs $\G$ of girth at least 5, we next consider graphs of girth 4. In light of Lemma~\ref{tri-free-all-comms}, it seems likely that many graphs of girth 4 will be RA, but Example~\ref{eg:not-RA} proves that not every graph of girth 4 is RA. We will later describe a family of graphs of girth 4 that are not RA. For now, let us use Lemma~\ref{tri-free-all-comms} to help describe what a non-RA graph of girth 4 would have to look like.

\begin{defi}
\label{def:prop-S}
    We say that the graph $\G$ satisfies the \emph{square completion property} if every path with three vertices can be completed to a 4-cycle. 
\end{defi}

\begin{prop} \label{girth-4}
    Suppose that $\Gamma$ is a connected graph with girth $4$. If $\G$ has at least one vertex of degree $2$, then $\G$ is RA. If $\G$ does not satisfy the square completion property, or if  $\Gamma$ has at least one vertex of degree $1$, then $\Gamma$ is strongly RA.
\end{prop} 

\begin{proof}
    If $\G$ has a vertex of degree $2$, then Lemma~\ref{corner-trick} and Lemma~\ref{tri-free-all-comms} imply that $\Comm(G, \G) = [G, G]^{|\G|}$, and so $\G$ is RA.

    If $\Gamma$ does not satisfy the square completion property, then there is a path with three vertices $u, v, w$ that cannot be completed to a $4$-cycle, which implies that $v$ is the only common neighbor of $u$ and $w$. Then $\Commd(G, \G) = [G, G]^{|\G|}$ by Lemma~\ref{threeinarow} and Lemma~\ref{tri-free-all-comms}, so $\G$ is strongly RA. In particular, if $\G$ has a vertex of degree 1, then it cannot satisfy the square completion property and thus is strongly RA.
\end{proof}

\begin{Example} \label{prop-s-example}
    Cube graphs $Q_d$ with $d \geq 2$ have the square completion property, as do complete bipartite graphs $K_{m,n}$. An $m \times k$ grid graph with $m \geq 2$ and $k \geq 3$ does not. 
\end{Example}

Let us now examine some consequences of Proposition~\ref{girth-4} for some families of graphs of girth 4.

\begin{theorem}
\label{thm:RA-families}
\begin{enumerate}[(a)]
    \item $C_4$ is RA but not strongly RA. 
    \item An $m \times k$ grid graph with $m \geq 2$ and $k \geq 3$ is strongly RA.
    \item The complete bipartite graph $K_{m,n}$ is RA, and is strongly RA if and only if $m$ and $n$ are relatively prime. 
\end{enumerate}
\end{theorem}

\begin{proof}
(a) This follows from Example~\ref{eg:RA-graph}.

(b) This follows from Proposition~\ref{girth-4} and Example~\ref{prop-s-example}.


(c) If $u$ and $v$ are vertices in the same set of the bipartition of $K_{m,n}$, then
\[ [h^{u}, g^{v}] [g^{u}, h^{u}] = [g,h]^{\bbe_u},\]
and then Lemma~\ref{tri-free-all-comms} shows that $K_{m,n}$ is RA. Note that since $K_{m,n}$ satisfies the square completion property, Proposition~\ref{girth-4} does not apply. Now, to show that $K_{m,n}$ is strongly RA, we have to only consider the effect of commutators $[g^u, h^v]$ with $u \neq v$. Let us assume without loss of generality that $m \leq n$. Let the bipartition of $K_{m,n}$ consist of sets $A$ and $B$ with $|A| = m$ and $|B| = n$. Then if $u$ and $v$ are distinct vertices of $A$ (resp. $B$), the element $[g^u, h^v]$ consists of $[g,h]$ on every element of $A$ (resp. $B$). If $u \in A$ and $v \in B$, then $[g^u, h^v]$ consists of $[g,h]$ on $u$ and $v$ only.

Suppose that $m$ and $n$ have a nontrivial common factor $k$, and let $G = D_8 = \langle x, y \mid x^2 = y^2 = (xy)^4 = 1 \rangle$.
Then the only nontrivial commutator is $[x,y]$. Since $[x,y]$ has order 2, we find that any of the generating moves preserve the relation that the number of vertices of $A$ that has state $[x,y]$ is congruent modulo $k$ to the number of vertices of $B$ with state $[x,y]$. In particular, we cannot have $[x,y]$ on just a single vertex. This shows that $K_{m,n}$ is not strongly RA in this case.

Now suppose that $m$ and $n$ are relatively prime, so that $m < n$. Fix elements $x,y \in G$ and let $c = [x,y]$. Then we can put $c$ on $n-m$ vertices of $B$ by first putting it on every vertex, (using $[x^u, y^v]$ with a $u$ and $v$ in $A$, and the repeating for a $u$ and $v$ in $B$), and then picking $m$ disjoint pairs $(a,b)$ with $a \in A, b \in B$ to multiply by $c^{-1}$. If $n-m$ is still greater than $m$, we can repeat this process, putting $c$ on every vertex of $A$ and picking $m$ more disjoint pairs to end up with $c$ on $n-2m$ elements of $B$. Continuing in this way, we eventually get $c$ on $n \mod m$ elements of $B$. Now we repeat the process with the roles of $A$ and $B$ switched. We are essentially just executing the Euclidean algorithm for finding $\gcd(m,n)$, and since $\gcd(m,n) = 1$, this will end with $c$ on a single vertex. The result then follows from Lemma~\ref{tri-free-all-comms}.
\end{proof}

We have seen that the cube $Q_2 = C_4$ is RA but not strongly RA. Furthermore, $Q_3$ is not even RA. Propositions~\ref{girth-4} and \ref{prop-s-example} suggest that cubes are worth further study as possible examples of graphs that are not RA. For now, let us generalize Theorem~\ref{thm:RA-families}(a).

    \begin{prop} \label{even-cube-RA}
    The cube graph $\G = Q_d$ is never strongly RA, but it is nonetheless RA if $d$ is even.
    \end{prop}

    \begin{proof}
    To show that $Q_d$ is not strongly RA, consider $G = D_8 = \langle x, y \rangle$, and let $c = [x,y]$ be the single nontrivial commutator. If $u \neq v$, then $[x^u, y^v]$ is either trivial, or it multiplies two coordinates by $c$ (either corresponding to $u$ and $v$, or to their two common neighbors). Since $c$ has order 2, the sum of powers of $c$ over all the vertices always remains even a click with a generator of $\Commd(D_8, Q_d)$, so we cannot put $c$ on a single vertex alone, and thus $Q_d$ is not strongly RA.
    
    Now let $d$ be even, let $v_0 = 00\cdots0$, and for $1 \leq i \leq d$, let $v_i$ be the neighbor of $v_0$ with a $1$ in the $i$th coordinate. Then for each $1 \leq k \leq d/2$, $v_{2k-1}$ and $v_{2k}$ have two common neighbors; $v_0$ and another vertex $u_k$. Then $[g^{v_0}, h^{u_k}]$ consists of $[g,h]$ only at $v_{2k-1}$ and $v_{2k}$. Thus,
    \[ [g^{v_0}, h^{u_1}] \cdots [g^{v_0}, h^{u_{d/2}}] \cdot [h^{v_0}, g^{v_0}] \]
    consists of $[h,g]$ only at $v_0$. Then Lemma~\ref{tri-free-all-comms} implies that $\G$ is RA.
    \end{proof}

Proposition~\ref{even-cube-RA} can be generalized to certain strongly regular graphs and certain box products of complete graphs, and we will explore this in \cite{FullPaper}.

Let us briefly consider graphs of girth 3. In this case, since we can no longer apply Lemma~\ref{tri-free-all-comms}, we no longer have an easy way to translate local properties to global properties. Still, highly-symmetric graphs with triangles are amenable to analysis.

\begin{theorem} \label{thm:girth3-RA-graphs}
The following graphs are RA:
\begin{enumerate}[(a)]
    \item The wheel graph $W_n$ with $n \geq 5$,
    \item The graph with vertices labeled $1, 2, \dots, n$ with $n > 4$ and edges described by triangles $123, 234, \dots, (n-2)(n-1)n$.
\end{enumerate}
\end{theorem}

\begin{proof}
    (a) Label the central vertex of $W_n$ as $v_n$ and the rest, say clockwise, as $v_1, v_2, \dots, v_{n-1}$. Then $[h^{v_2}, g^{v_3}] \cdot [g^{v_2}, h^{v_2}] = [g,h]^{\bbe_1}$. Similarly, $[g,h]^{\bbe_i} \in [G^\Gamma, G^\Gamma]$ for $i = 2, \ldots, n-1$. Finally, since $[g,h]^{\bbe_1 + \cdots + \bbe_{n-1}} [h,g]^{v_n} = [h,g]^{\bbe_n}$, the result follows.

    (b) Note $[h^{v_2}, g^{v_4}][g^{v_2}, h^{v_2}] = [g, h]^{\bbe_1} \in G^\G$, and $[h^{v_3}, g^{v_5}][g^{v_3}, h^{v_3}] = [g, h]^{\bbe_2} \in G^\G$, so we can put any commutator in vertices 1 or 2 individually, and an easy strong induction argument shows we can put one in any individual vertex from there. 
\end{proof}

For any RA graph -- including all of the graphs in Theorem~\ref{thm:RA-families}, Proposition~\ref{even-cube-RA}, and Theorem~\ref{thm:girth3-RA-graphs} -- we can understand $G^\G$ just by understanding $(\Gab)^\G$. So far, we have only encountered one graph that is not RA (namely, $Q_3$ -- see Example~\ref{eg:not-RA}). Let us turn our attention to finding more.

\section{Computing whether a graph is RA}
\label{compute-RA}

\subsection{The RA Matrix}
\label{sec:RA-matrix}


    How can we determine whether a graph $\Gamma$ is RA? 
    By Proposition~\ref{prop:indistinguishable}, we can first make sure we're working with a graph that is connected and neighborhood-distinguishable. Then, we might compute the girth and try to apply Corollary~\ref{cor:girth-5} or Proposition~\ref{girth-4}. This leaves out graphs with 3-cycles and certain graphs of girth 4. What techniques remain to determine whether such a graph is RA?

   If we can show that elements of the form $[g^u, h^v] \in [G^\G, G^\G]$ actually generate all of $[G, G]^{|\G|}$, i.e. if $\Comm(G, \G) = [G, G]^{|\G|}$, for every $G$, then we know $\G$ is RA. To do so, we need to consider the set of indicator vectors for $B(u) \cap B(v)$. Just as in the proof of Proposition~\ref{Z-AW-universal}, if such vectors span $\Z^{|\G|}$ then the elements $[g^u, h^v]$ would generate all of $[G, G]^{|\G|}$. This gives us a new practical tool to help test whether a graph is RA.

    \begin{defi}
        The \emph{RA Matrix} $C_\G$ of $\G$ is the matrix whose rows are the indicator vectors for $B(v_i) \cap B(v_j)$ where $v_i$ and $v_j$ range over all the vertices of $\G$.
    \end{defi}

    The above discussion gives us
    
    \begin{theorem}
    \label{thm:CG}
        If $\Z^{C_\G} = \Z^{|\G|}$, i.e. if the $\Z$-span of $C_\G$ is all of $\Z^{|\G|}$, then $\G$ is RA.
    \end{theorem}
    
    Note that, if $\Z^{C_\G} \neq \Z^{|\G|}$, then that only establishes that for some groups $G$ we may have $\Comm(G, \G) \neq [G,G]^{|\G|}$. It could still be the case that $\G$ is RA. For now though, we will only use Theorem~\ref{thm:CG} to verify certain graphs as RA and to find promising candidates of non-RA graphs. Let us start with some computational data.

    \begin{prop} \label{7-verts-RA}
    Every (connected, neighborhood-distinguishable) graph on up to $7$ vertices is RA.
    \end{prop}

    \begin{proof}
    Using the graph database in SageMath, we found that $\Z^{C_\G} = \Z^{|\G|}$ for all eligible graphs with at most 7 vertices. (See Figure~\ref{fig:sage-small-RA-code} in \ref{sec:app-code})
    \end{proof}

    The bound on Proposition~\ref{7-verts-RA} is as sharp as possible, since we already determined that the cube graph $Q_3$ (with 8 vertices) is not RA (see Example~\ref{eg:not-RA}). 

    If we are searching for graphs that are not RA, then we can restrict our attention to graphs whose RA Matrix has nontrivial elementary divisors. Sometimes, the structure of the graph makes it clear that this will occur.

    \begin{prop} \label{pqr-not-RA-mat}
        Let $p$ be a prime number. Suppose that every vertex of $\G$ has degree congruent to $-1 \pmod p$, that every pair of adjacent vertices of $\G$ has a number of common neighbors that is $-2 \pmod p$, and that every pair of vertices of $\G$ that are distance 2 apart have a number of common neighbors that is divisible by $p$. Then $C_\G$ 
        has an elementary divisor that is a multiple of $p$.
    \end{prop} 
    \begin{proof}
        There are three types of rows of $C_\G$. 
        The rows that correspond to
        neighborhoods of single vertices have row sum $d+1$ for vertices of degree $d$. 
        The rows corresponding to $B(u) \cap B(v)$ where $u$ and $v$ are adjacent have row sum $q+2$ if they have $q$ common neighbors. 
        The rows corresponding to $B(u) \cap B(v)$ where $u$ and $v$ are distance two apart have row sum $r$ equal to how many common neighbors they have. The given conditions ensure that each of these sums is divisible by $p$, which ensures that the matrix has an elementary divisor that is divisible by $p$ (see Proposition~\ref{row-sums-fund-factors}).
    \end{proof}

\subsection{The sufficiency of the RA Matrix}
\label{SuffRAMatrix}
    After Theorem~\ref{thm:CG}, we remarked that we had not established the converse. Certainly, there are some groups $G$ where the RA matrix does not tell us the entire story about $\CComm(G, \G)$ when $\G$ is not $G$-RA. However, we will see that the RA matrix does suffice to tell us whether a group is $G$-RA for every Heisenberg group, and this will be enough to determine whether a graph is RA over every group or not.

    Let us start by considering a more general and somewhat more difficult question. Recall the chain of subgroups described in \eqref{all-layers}. Let us focus on the containment $[G^\G, G^\G] \leq \CComm(G, \G)$. Our question is: could it be that these groups are always equal?

    First, consider some examples where we replace $\G$ with a matrix $M$ and ask the analogous question.

    \begin{Example} \label{mat-2}
        Let $M = [2]$ and $G = D_8 = \langle r, s \mid r^4 = s^2 = 1, srs = r^{-1}\rangle$. Then $r^2$ is the only nontrivial commutator in $D_8$ and clicking with $r$ gives us $(r^2) \in G^M$; indeed, $G^M = \{(r^2), (1)\}$ since clicking with any element gives the square of that element and the only squares are $r^2$ and $1$, so $[G^M, G^M] = 1 \subsetneqq [G, G]^1 \cap G^M$.  
    \end{Example}

    \begin{Example} \label{mat-neg-one}
        Let
        \[M = \mat111{-1}.\]
        Then we can explicitly compute that $[D_8^M, D_8^M] = \{(r^2, r^2), (1, 1)\}$ but $r^{v_1}r^{v_2} = (r^2, 1) \in D_8^M$ so again $[G^M, G^M] \subsetneqq [G, G]^2 \cap G^M$. 
    \end{Example}

    Thus, using matrices with anything other than 0's and 1's is liable to produce examples like the above with $[G^M, G^M] \subsetneqq [G, G]^n \cap G^M$. We have already seen a small hint that matrices with only 0's and 1's may act in a special way; consider Proposition~\ref{basic-props}(c). In our current context, the point is that $[g^v, h^v] = [g,h]^v$, whereas this equation need not hold if we replace $v$ with a general vector $\bx$. 

    We will broaden our scope slightly though, to allow any set of $\{0,1\}$-vectors.
    In order to tease out the difference between $[G^S, G^S]$ and $\CComm(G,S)$, note that working modulo $[G^S, G^S]$ essentially lets you rearrange your clicks. For example,
    \[ g_1^{\bx_1} g_2^{\bx_2} [G^S, G^S] = g_2^{\bx_2} g_1^{\bx_1} [G^S, G^S]. \]

    Let us make a slightly technical definition.

    \begin{defi}
        We say $G$ has \emph{faithful abelian generators} if $G^\Ab \cong \bigoplus_{\alpha=1}^k \Z/r_\alpha\Z$ and for each $\alpha$ there is a representative $y_\alpha$ for $\overline{\bf e_\alpha} \in G^\Ab$ so that the order of $y_\alpha$ in $G$ equals $r_\alpha$.
    \end{defi}

    Note that by definition, for each $\alpha$ there always exists a representative $y$ for $\overline{\bf e_\alpha}$ such that $y^\alpha$ is in $[G, G]$; what we need is for $y^\alpha$ to actually be the identity element of $G$. Not every group $G$ has faithful abelian generators, and in fact the requirements of the next theorem can be lessened somewhat \cite{FullPaper}, but many of the common groups do have faithful abelian generators; symmetric, dihedral, and Heisenberg groups are only a few examples.

   \begin{theorem}
   \label{S1=S2-general}
        Let $S \subset \{0,1\}^n$. Let $G$ be any group with faithful abelian generators. Then $[G^S, G^S] = \CComm(G, S)$.
    \end{theorem}
    \begin{proof}
        By definition $[G^S, G^S] \le \CComm(G, S)$, so we just need to show the opposite containment. Since $[G^S, G^S]\lhd G^S$, we will take a coset of $[G^S, G^S]$ represented by an element of $\CComm(G, S) = [G, G]^n \cap G^S$. If $S = \{\bx_1, \bx_2, \dots, \bx_s\}$, such an element is a product of clicks $g_i^{\bx_i}$, and modulo $[G^S, G^S]$, we can rearrange the order of such clicks, so a new equivalent representative can be written as 
        \begin{align*}
        g_{11}^{\bx_1}g_{12}^{\bx_1} \cdots g_{1,t_1}^{\bx_1} \cdots g_{s,1}^{\bx_s}g_{s,2}^{\bx_s} \cdots g_{s,t_s}^{\bx_s} \\
        = (g_{11}g_{12} \cdots g_{1,t_1})^{\bx_1} \cdots (g_{s,1}g_{s,2} \cdots g_{s,t_s})^{\bx_s}    
        \end{align*}
        (using the fact that each vector $\bx_i$ has only 0's and 1's in it; see Proposition~\ref{basic-props}). 
        Let $g_i = g_{i1}g_{i2}\cdots g_{i, t_i}$ and write the representative of the coset of $[G^S, G^S]$ as $g_1^{\bx_1} \cdots g_s^{\bx_s}[G^S, G^S]$. Write $G^\Ab \cong \bigoplus_{\alpha=1}^k \Z/r_\alpha\Z$ and for each $\alpha$, pick a representative $y_\alpha$ for $\overline{\bf e_\alpha} \in G^\Ab$ so that the order of $y_\alpha$ in $G$ equals $r_\alpha$ (this is the hypothesis about $G$). 
        Now, for each $i$, if the image $\overline{g_i}$ of $g_i$ in $\Gab$ is $(q_{i, 1}, q_{i, 2}, \dots, q_{i, k})$ then we can write $g_i = \left(\prod_{\alpha=1}^k y_\alpha^{q_{i, \alpha}}\right)c_i$ for some  $c_i \in [G, G]$. Since $\prod_{i=1}^s g_i^{\bx_i}[G^S, G^S] \in [G, G]^n$, for each coordinate $j = 1, 2, \dots, n$ we have $\overline{\prod_{i=1}^s g_i^{(\bx_i)_j}} = \bar 1 \in G/[G, G]$ so for any fixed coordinate $j = 1, \dots, n$ and any fixed summand $\alpha = 1, \dots, k$ of $G^\Ab$ we have $\prod_{i=1}^s y_\alpha^{q_{i, \alpha}(\bx_{i})_j} \in [G, G]$. In other words, $\sum_{i=1}^s q_{i, \alpha}(\bx_i)_j \equiv 0 \pmod {r_\alpha}$ so $\prod_{i=1}^s y_\alpha^{q_{i, \alpha}(\bx_{i})_j} = 1 \in G$ and, putting these together, for each $\alpha$ we have $\prod_{i=1}^s (y_\alpha^{q_{i, \alpha}})^{\bx_i} = (1)_{j=1}^n \in G^n$. 

        But that means, modulo $[G^S, G^S]$, we can rewrite the representative
        \[\prod_{i=1}^s \left(\prod_{\alpha=1}^k y_\alpha^{q_{i, \alpha}} c_i\right)^{\bx_i}\]
        of the coset of $[G^S, G^S]$ as
        \[(\prod_{i=1}^s (y_1^{q_{i, 1}})^{\bx_i}) (\prod_{i=1}^s (y_2^{q_{i, 2}})^{\bx_i}) \cdots (\prod_{i=1}^s (y_k^{q_{i, k}})^{\bx_i})(\prod_{i=1}^s c_i^{\bx_i})\]
        and each of the first $k$ products is simply $(1)_{j=1}^n$, while the last product is in $[G^S, G^S]$, so the original coset of $[G^S, G^S]$ is the identity coset.
    \end{proof}

    \begin{coro}
        \label{S1=S2}
        Let $S \subset \{0, 1\}^n$. For any symmetric, dihedral, or Heisenberg group $G$, we have $\CComm(G, S) = [G^S, G^S]$.
    \end{coro}
    \begin{proof}
        The abelianization of a symmetric group $S_m^\Ab \cong \Z/2\Z$ is generated by $\overline{(12)}$. Any dihedral group $G = D_{2k}$ has $G^\Ab \cong \Z/2\Z$,  or $\Z/2\Z \times \Z/2\Z$ and in either case we can take the representative of each coordinate to be a flip. The abelianization of any Heisenberg group of order $p^{2n+1}$ is a direct sum of $2n$ copies of $\Z/p\Z$'s, and the group has exponent $p$ for $p > 2$ so any representative of any coset will have order $p$. For $p = 2$, we can take the representatives to be elements of order 2 (just let $\vec a = \vec 0$ or $\vec b = \vec 0$ and set the upper-right element to 1). If $G \cong [G, G] \oplus G^\Ab$, there is a natural monomorphism from $G^\Ab$ to $[G, G]$ under which the generators of $G^\Ab \cong \Z/r_1\Z \oplus \cdots \oplus \Z/r_k\Z$ go to elements of $G$ of the same order.
    \end{proof}

    \begin{theorem}
        \label{thm:RA-over-Heisenberg}
        \begin{enumerate}[(1)]
            \item If $G$ is any Heisenberg group $H(\F_p)$ and $\G$ is any graph, then $\CComm(G,\G) = [G^\G, G^\G] = \Comm(G, \G) = \langle [g^u, h^v] \mid g, h \in G, u, v \in \G \rangle$.
            \item For every prime $p$, $\G$ is RA over $H(\F_p)$ if and only if $C_\G$ has full rank modulo $p$.
            \item If the graph $\G$ is RA over every Heisenberg group, then $C_\G$ has full rank over $\Z$ and $\Comm(G, \G) = [G, G]^n$ for every group $G$; in particular, $\G$ is RA over every group. 
        \end{enumerate}        
    \end{theorem}    
    \begin{proof}
        (1) We have shown the first equality in Corollary~\ref{S1=S2}, and the last is just the definition of $\Comm(G, \G)$. The middle equality follows from $[G^\G, G^\G]$ being the normal closure of $\Comm(G, \G)$ and the fact that the commutator subgroup of every Heisenberg group is central.

        (2) Let $G$ be a Heisenberg group of order $p^3$. Note that $[G,G]^n \cong (\Z_p)^n = \langle c \rangle^n$ for a fixed commutator $c$. Since $[g^u, h^v] = [g,h]^{B(u) \cap B(v)}$ and $[g,h]$ is a power of $c$, $\Comm(G,\G) = (\Z_p)^{C_\G}$. 

        By (1), $\CComm(G,\G) = \Comm(G,\G)$. $\G$ is RA over $G$ if and only if $[G,G]^n \sub \CComm(G,\G)$. That is, if and only if $(\Z_p)^n \sub (\Z_p)^{C_\G}$; i.e. $C_\G$ has full rank modulo $p$.
        
        (3) By part (2), if $\G$ is RA over every Heisenberg group, then $C_\G$ is full rank modulo every prime $p$, so $\Z^{C_\G} = \Z^{|\G|}$ (no  elementary divisors are divisible by any prime, 
        hence are just 1), which implies the result by Theorem~\ref{thm:CG}.
    \end{proof}

    We now can prove our main theorem of this section.

    \begin{theorem} \label{RA-iff}
    $\G$ is RA if and only if $\Z^{C_\G} = \Z^{|\G|}$.
    \end{theorem}

    \begin{proof}
    Theorem~\ref{thm:CG} proves one direction. For the other, suppose that $\G$ is RA. Then it is RA over every Heisenberg group, and Theorem~\ref{thm:RA-over-Heisenberg} proves the claim.
    \end{proof}

    Theorem~\ref{RA-iff} gives a  practical way of determining if a graph is RA. The RA matrix $C_\G$ may have some elementary divisors, but by the formula $\frac{d_i(C_\G)}{d_{i-1}(C_\G)}$ for the $i$-th elementary divisor, since $C_\G$ contains the rows of the activation matrix $A_\G$ for $\G$, each $d_i(C_\G)$ divides $d_i(A_\G)$, and therefore any primes dividing the elementary divisors of $C_\G$ must also divide the largest elementary divisor $d$ of $A_\G$. That means one can look at just the primes $p$ dividing $d$, compute the rank of $C_\G \pmod p$ for each such prime $p$, and if the rank turns out to be full for all such primes, we know the graph is RA. Theorem~\ref{RA-iff} also tells us that there is no hidden information between the layers $\Comm(G, \G), [G^\G, G^\G]$, and $\CComm(G,\G)$ that may make it difficult to determine if the graph is RA. It is enough to determine whether $\Z^{C_\G} = \Z^{|\G|}$. 

    Considering Proposition~\ref{pqr-not-RA-mat}, we get the following corollary of Theorem~\ref{RA-iff}.

    \begin{coro} \label{pqr-not-RA}
    Let $p$ be a prime number. Suppose that every vertex of $\G$ has degree congruent to $-1 \pmod p$, that every pair of adjacent vertices of $\G$ has a number of common neighbors that is $-2 \pmod p$, and that every pair of vertices of $\G$ that are distance 2 apart have a number of common neighbors that is divisible by $p$. Then $\G$ is not RA over $H(\F_p)$. 
    \end{coro} 

    In particular, letting $p = 2$, we have the following result. 

    \begin{coro} \label{girth-4-not-RA}
    If $\G$ has girth 4, every vertex has odd degree, and every pair of vertices at a distance 2 apart has an even number of common neighbors, then $\G$ is not $D_8$-RA.
    \end{coro}


    We can now present our first families of graphs that are not RA.

    \begin{coro} \label{graphs-not-RA}
    The following graphs are not $D_8$-RA:
    \begin{enumerate}[(a)]
        \item The cube $Q_d$ with $d$ odd.
        \item The folded cube $\square_d$ with $2^{d-1}$ vertices, where $d$ is odd or 2.
    \end{enumerate}       
    \end{coro}

\begin{proof}
    Part (a) follows directly from Corollary~\ref{girth-4-not-RA}, since vertices that are distance 2 apart have exactly two common neighbors.



    For part (b), we will think of the folded cube $\square_d$ as represented by the hypercube $Q_{d-1}$ where opposite vertices are connected (those that are distance $d-1$ apart). If $d = 2$ or 3, the graph is complete ($K_2$ or $K_4$), respectively, and those graphs have neighborhood-indistinguishable vertices, so let's assume $d > 3$. This graph has girth 4 and each vertex has odd degree $d$. If two vertices are a distance 2 apart then they have 2 neighbors in common. Indeed, either they are connected by changing two bits in their binary representation (in which case they have the same neighbors as $Q_{d-1}$) or they are connected by an edge joining opposite vertices and one more edge representing a single bit flip, which can be done in either order, thus giving exactly 2 neighbors in common, and Corollary~\ref{girth-4-not-RA} finishes the proof.
\end{proof}

\section{Conclusion}

Our work has built on previous work generalizing the classic 2-color $5 \times 5$ Lights Out Puzzle by extending the possibility of the color space to elements of a group $G$. As a result, we have a new notion of the ``graph product'' of a group, which not only generalizes the diagonal subgroup of a product of groups but also allows for the invention of a slew of new puzzles. Any puzzle that involves some  permuting (like the cubies of a Rubik's Cube) could be put into the vertices of a graph and the puzzles could be connected in a way that makes it so when you perform an action on one puzzle, the puzzles in adjacent vertices have the same action done on them. The theory developed here can tell you which starting states would be solvable in such a setting, and the authors would be thrilled to hear of interesting scenarios that the reader might invent.

We have largely focused on applications where the graph in question is simple and undirected, but it is possible to consider directed graphs as well as multigraphs and graphs with multiple self-loops. There is likely more to explore with particular families of simple undirected and more general graphs. 

One particularly interesting scenario to explore is when a vertex does not affect itself; for instance, in the seminal work \cite{Sutner1988sigma, Sutner2000sigma} Sutner deduces that it is rather simple to figure out the size of $G^\G$ when $G = \Z_2$ and $\G$ is an $m \times n$ grid if clicking a vertex \textit{does not} affect itself, but if a vertex \textit{does} affect itself, it is complicated and involves finding the degree of the $\gcd$ of two polynomials.

We have shown that many graphs are RA (those with girth $\ge 5$ and many others, see the various results in Section \ref{RA-families}), and we have a characterization of RA graphs in Theorem \ref{RA-iff}, but one clear possibility for future research remains in exploring other families of graphs and finding families that either are or are not RA. There is also more to be said about the structure of $G^\G$ when $\G$ is not RA. We will explore some of these ideas in \cite{FullPaper}.

\appendix
\renewcommand{\thesection}{Appendix \Alph{section}.}
\section{Explicit Description of $G^\G$}

We have only described $G^\G$ via the short exact sequence 
\[1 \to [G, G]^n \cap G^\G \to G^\G \to (\Gab)^n\]
for now, which allows us to have a rough description of $G^\G$ and, for instance, to know its size. If we want to actually write down elements in $G^\G$, we can reduce $A_\G$ to an echelon form by using only two row operations: adding a multiple of one row to another and swapping rows. We will find that $A_\G$ is row-equivalent to a matrix $H$ in \emph{Hermite normal form}, whose entries are integers, such that
\begin{enumerate}
    \item $H$ is upper triangular,
    \item Every pivot (leading coefficient) is positive and is to the right of the pivots of the rows above it, and
    \item The elements below pivots are zero, and the elements above the pivots are nonnegative numbers that are smaller than the pivots.
\end{enumerate}
In fact, $H = UA_\G$ where $U$ is a unimodular matrix, i.e. $\det U \in \{\pm 1\}$.

In our context, we can also allow renumbering the coordinates, which rearranges the columns, so $H$ can end up in the nice form
\begin{equation}
\label{nice-HNF}
\begin{pNiceArray}{c|c|c}
I_r & * & * \\ \hline 
0 & T & * \\ \hline
0 & 0 & 0
\end{pNiceArray},
\end{equation}
where $I_r$ is an $r \times r$ identity matrix and $T$ is upper triangular of size $k \times k$ with pivots $a_1 \le \cdots \le a_k$ along the diagonal of $T$ (with $a_1 > 1$).

The first $r + k$ resulting rows, say $\bx_1, \ldots, \bx_{r+k}$ of $H$, are a nicer basis of $\Z^\G$ to work with than the original rows of $A_\G$ because now we can see that for every $g_1, \ldots, g_r \in G$,
\[ g_1^{\bx_1} \cdots g_r^{\bx_r} = (g_1, \ldots, g_r, *, \ldots, *). \]
Thus, just as in $\Z^\G$, we have complete freedom in $G^\G$ to set the first $r$ coordinates. For coordinates $r+i$ with $1 \leq i \leq k$, the fact that we can change them by any multiple of $a_i$ in $\Z^H$ tells us that in $G^\G$, we can multiply them (sequentially) by any element $g^{a_i}$. Note that, depending on $G$, the effect of this can be drastically different. 
For example, if $a_i = 3$ and $G$ has exponent 3, then we cannot change coordinate $r+i$ by clicking coordinate $r+i$ itself, and the only way to change it may be to click coordinates $r+j$ with $1 \le j < i$. 
On the other hand, if the cubes of $G$ generate all of $G$, then we can change coordinate $r+i$ any way we wish by clicking that coordinate itself. 

In particular, for a divisible group $G$, or a group that is $p$-divisible for every prime $p$ dividing the nonzero elementary divisors of $A_\G$, we can set the first $r+k$ coordinates however we wish.
If $G$ is abelian, then just as with $\Z^H$, it is not possible to change the remaining $n-r-k$ coordinates (corresponding to the zeroes on the main diagonal of $H$) without affecting the previous coordinates. For nonabelian groups, this is not necessarily the case: if $\G$ is $G$-RA, then we can nudge these coordinates by any commutator in $G$. 

\textbf{Example.}
Let $G$ be abelian, and let $\G = C_4$. Then we can reduce $A_\G$ to
\[
H = \begin{pmatrix}
1 & 0 & 0 & 2 \\
0 & 1 & 0 & 2 \\
0 & 0 & 1 & 2 \\
0 & 0 & 0 & 3
\end{pmatrix}.
\]
Then, writing $\bx_1, \ldots, \bx_4$ for the rows of $H$, a typical element of $G^\G$ is
\[ g_1^{\bx_1} \cdots g_4^{\bx_4} = (g_1, g_2, g_3, g_1^{2} g_2^{2} g_3^{2} g_4^3). \]
In fact, every element of $G^\G$ has this form. 

Let us point out why we need $G$ abelian in the previous example. If we want an element of the form $(g_1, g_2, g_3, *)$, then we can proceed as in the example, or we could switch the order of some of the elements:
\[ g_2^{\bx_2} g_1^{\bx_1} g_3^{\bx_3} g_4^{\bx_4} = (g_1, g_2, g_3, g_2^{2} g_1^{2} g_3^{2} g_4^3) .\]
Note that the last coordinate may be different from before if $G$ is nonabelian, but that it does differ by a commutator in this case. 

\section{Code for graph products}
\label{sec:app-code}

The first block of code below, in Figure \ref{fig:G-Gamma-code}, can be used to compute $G^\G$ in GAP.

\begin{figure}[hbtp]
\begin{small}
\begin{verbatim}
# Gamma is a list of lists, representing a graph 
# via adjacency lists. The vertices are assumed 
# to be {1, ..., Size(Gamma)}
GraphProduct := function(G, Gamma)

local gens, D, embeds, newgens, g, i, S, k, nbhd,
newgen, imgs;
	
k := Size(Gamma);
gens := GeneratorsOfGroup(G);
D := DirectProduct(List([1..k], i -> G));
embeds := List([1..k], i -> Embedding(D, i));
newgens := [];

for g in gens do
for i in [1..k] do
nbhd := Concatenation(Gamma[i], [i]);
imgs := List(nbhd, i -> Image(embeds[i], g));
newgen := Product(imgs);
Add(newgens, newgen);
od;
od;

S := Subgroup(D, newgens);
return S;
end;
\end{verbatim}
\end{small}
    \caption{GAP code to compute $G^\G$}
    \label{fig:G-Gamma-code}
\end{figure}

Figure \ref{fig:Q3-RA-code} shows code that compute the index of $\CComm(G,\G)$ in $[G,G]^{|\G|}$ in GAP.

\begin{figure}[hbtp]
\begin{small}
\begin{verbatim}
# Returns the index of Comm(G, Gam) in [G, G]^n
RAIndex := function(G, Gam)
	local n, dg;
	
	n := Size(Gam);
	dg := DerivedSubgroup(G);
	
	return Size(dg)^n * Size(GraphProduct(G/dg, Gam)) /
    Size(GraphProduct(G, Gam));
	end;

Q3 := [[2,4,5], [1,3,6], [2,4,7], [1,3,8],
[1,6,8], [2,5,7], [3,6,8], [4,5,7]];

# Show that Q_3 is D10-RA but not D8-RA
RAIndex(DihedralGroup(10), Q3);
RAIndex(DihedralGroup(8), Q3);
\end{verbatim}
\end{small}
    \caption{GAP code to show that $Q_3$ is $D_{10}$-RA but not $D_8$-RA.}
    \label{fig:Q3-RA-code}
\end{figure}

The last block of code, in Figure \ref{fig:sage-small-RA-code}, uses SageMath to show that all connected, neighborhood-distinguishable graphs on up to 7 vertices are RA.

\begin{figure}[hbtp]
\begin{small}
\begin{verbatim}
# Do all vertices have unique closed neighborhoods?
def allNeighborhoodsDistinct(g):
    n = g.order()
    M = g.adjacency_matrix() + matrix.identity(n)
    return len(Set(M.rows())) == n

# Find the Smith normal form of A+I
def snf(g):
    n = g.order()
    M = g.adjacency_matrix() + matrix.identity(n)
    Ms = M.smith_form()[0]
    return Ms

# Build a matrix of all intersections of closed neighborhoods
def RAMat(g):
    n = g.order()
    M = g.adjacency_matrix() + matrix.identity(n)
    cols = M.columns()
    newcols = []
    for c1 in cols:
        for c2 in cols:
            newcol = [c1[i]*c2[i] for i in range(len(c1))]
            newcols.append(newcol)
    M2 = matrix(newcols).transpose()
    return M2

def isRA(g):
    n = g.order()
    M2 = RAMat(g)
    M2s = M2.smith_form()[0]
    return all([M2s[i][i] == 1 for i in range(n)])

graphLists = []

# This contains all graphs with 7 or fewer vertices
D = GraphDatabase()
for i in range(3,8):
    S = GraphQuery(D, display_cols=['graph6'], num_vertices=i)
    L = S.get_graphs_list()
    graphLists.append([g for g in L if g.is_connected() and
    allNeighborhoodsDistinct(g)])

print("Number of neighborhood-distinct graphs that are RA:")
for i in range(3,8):
 print(len([g for g in graphLists[i-3] if isRA(g)]), 
          "out of", len([g for g in graphLists[i-3]]), 
          "with", i, "vertices")
    \end{verbatim}
    \end{small}
    \caption{SageMath code to check all graphs with up to 7 vertices for the RA property.}
    \label{fig:sage-small-RA-code}
\end{figure}

\bibliographystyle{plain}
\bibliography{refs}
\end{document}